\documentclass[11pt,leqno]{amsart}
\usepackage{amsmath,amssymb,amsthm}
\newtheorem{theorem}{Theorem}[section]
\newtheorem{lemma}[theorem]{Lemma}

\newtheorem{proposition}[theorem]{Proposition}
\newtheorem{corollary}[theorem]{Corollary}
\theoremstyle{definition}

\theoremstyle{remark}
\newtheorem{remark}[theorem]{Remark}
\numberwithin{equation}{section}

\def\fnote#1{\footnote}

\def\natu{{\mathbb N}}

\def\real{{\mathbb R}}

\def\ignora#1{}
\def\lbl#1{\label{#1}}        
\def\n3#1{\left\vert  \! \left\vert \! \left\vert \, #1 \, \right\vert \!
  \right\vert \! \right\vert }

\begin{document}

\title{ Octahedral norms and convex combination of slices in Banach spaces  }

\author{Julio Becerra Guerrero, Gin{\'e}s L{\'o}pez-P{\'e}rez and Abraham Rueda Zoca}
\address{Universidad de Granada, Facultad de Ciencias.
Departamento de An\'{a}lisis Matem\'{a}tico, 18071-Granada
(Spain)} \email{juliobg@ugr.es, glopezp@ugr.es,
arz0001@correo.ugr.es}

\thanks{The first author was partially supported by MEC (Spain) Grant MTM2011-23843 and Junta de Andaluc\'{\i}a grants
FQM-0199, FQM-1215. The second author was partially supported by
MEC (Spain) Grant MTM2012-31755 and Junta de Andaluc\'{\i}a Grant
FQM-185.} 
\maketitle \markboth{J. Becerra, G. L\'{o}pez and A. Rueda   }{
  Octahedral norms }

\begin{abstract}
 We study the relation between octahedral norms, Daugavet property
 and the size of convex combinations of slices in Banach spaces. We
 prove that the norm of an arbitrary Banach space is octahedral
 if, and only if, every convex combination of $w^*$-slices in the
 dual unit ball has diameter  $2$, which answer an open question. As a consequence we get that the Banach
spaces with the Daugavet property and its dual spaces have
octahedral norms. Also, we show that for every separable Banach
space containing $\ell_1$ and for every $\varepsilon >0$ there
 is an equivalent norm so that every convex combination of $w^*$-slices in the
 dual unit ball has diameter at least $2-\varepsilon$.
\end{abstract}

\section{Introduction}
\par
\bigskip

We recall that a norm $\Vert\cdot\Vert$ of a Banach space $X$ is
called octahedral if for every $\varepsilon>0$ and for every
finite-dimensional subspace $Y$ of $X$   there is $x\in S_X$ such
that
$$ \Vert y+\lambda x\Vert\geq (1-\varepsilon)(\Vert
y\Vert+\vert\lambda\vert)$$ for every $y\in Y$ and
$\lambda\in\real$.

The octahedral norms were introduced by G. Godefroy in \cite{G},
where it is    proved that every Banach space containing $\ell_1$
can be equivalently renormed so that the new norm is octahedral.
Also some norms weaker than octahedral norms were used to
characterize Banach spaces containing $\ell_1$ \cite{D}. On the
other hand, R. Deville proved that  every convex combination of
$w^*$-slices of the unit ball in the dual of a Banach space $X$
has diameter two, whenever $X$ has an octahedral norm, and it is
asked  about the veracity of  the converse statement (see
\cite[Proposition 3, Remarks (c), pag. 119]{D}).

The aim of this note is to show in Theorem \ref{octahedral} that
the norm of a  Banach space $X$ is octahedral if, and only if,
every convex combination of $w^*$-slices in the unit ball of $X^*$
has diameter 2. Some consequences can be obtained from this fact.
For example, we get in Corollary \ref{octahedral-DG} that a Banach
space with the Daugavet property and its dual space have
octahedral norms. The fact that every Banach space with the
Daugavet property has an octahedral norm has been recently proved
in the separable case in \cite{KShW}. In the world of
$JB^*$-triple we get in Corollary \ref{octahedral-triples}, that a
real $JB^*$-triple $X$ has the Daugavet property if, and only if,
the norm of $X$ is octahedral, and in Corollary
\ref{octahedral-triples-dual}, it is shown that the dual of every
real $JB^*$-triple, has octahedral norm.

Finally, we prove in Proposition \ref{convex-w}, that for every
Banach  space $X$ containing $\ell_1$ and for every
$\varepsilon>0$ there is an equivalent  norm in $X$ such that
every convex combination of slices in the new unit ball of $X^*$
has diameter $2-\varepsilon$. This result has relation with the
following problem posed in \cite{G}: has every Banach space
containing $\ell_1$ an equivalent norm so that the  corresponding
bidual norm is octahedral?

We pass now to introduce some notation. $B_X$, respectively $S_X$,
stands for the unit ball, respectively unit sphere, of the Banach
space $X$. All Banach spaces considered will be real. By $w$ will
denotes the weak topology in $X$ and by $w^*$ the weak-star
topology in $X^*$. We recall that a slice in $X$ is a subset
defined by $S(B_X, f,\alpha)=\{x\in B_X: f(x)>1-\alpha\}$, where
$f\in S_{X^*}$ and $0<\alpha<1$. Similarly, a $w^*$-slice in $X^*$
is a subset defined by $S(B_{X^*},x,\alpha)=\{f\in
B_{X^*}:f(x)>1-\alpha\}$, where $x\in S_{X}$ and $0<\alpha<1$.

 \section{Main results}
\par
\bigskip

In \cite[Proposition 3 and Theorem 1]{D} it is proved that if a
Banach space $X$ has octahedral norm, then every convex
combinations of $w^*$-slices in $X^*$ has diameter 2, leaving open
the converse statement \cite[Remarks (c), pag. 119]{D}. Our first
goal is to prove the validity of this converse statement. For sake
of completeness we show the complete equivalence.

\begin{theorem}\lbl{octahedral}
Let $X$ be a Banach space. Then the following are equivalent:
\begin{enumerate}
\item The norm of $X$ is an octahedral norm. \item Every
  convex combination of $w^*$-slices in $B_{X^*}$ has diameter $2$.
\end{enumerate}
\end{theorem}
\begin{proof} i)$\Rightarrow$ii) Pick $N\in\mathbb N, x_1,\ldots, x_N\in S_X, \rho_1,\ldots,
\rho_N\in (0,1)$ and\break
 $\alpha_1 ,\ldots ,\alpha_N >0$ such
that $\sum_{i=1^N}\alpha_i=1$. Let $\rho:=\min\limits_{1\leq i\leq
N} \rho_i$. Then

$$\sum_{i=1}^N\alpha_i S( B_{X^*},x_i,\rho)\subseteq \sum_{i=1}^N\alpha_i S( B_{X^*},x_i,\rho_i)$$

So it is enough to prove that $diam\left (\sum_{i=1}^N\alpha_i S(
B_{X^*},x_i,\rho)\right )=2$.

Put $Y=span(\{x_1,\ldots, x_N\})$ and fix $n\in\mathbb N$. As
$\Vert\cdot\Vert$ is octahedral there exists $x_n\in S_X$
satisfying

$$\Vert y+\alpha x_n\Vert\geq \left (1-\frac{1}{n}\right ) (\Vert y \Vert+\vert \alpha\vert)$$

So

\begin{equation}\label{primeradesig}\Vert x_i\pm x_n\Vert\geq 2 \left (1-\frac{1}{n} \right )\ \ \ i\in \{1,\ldots, N\}\end{equation}

For $i\in \{1,\ldots, N\}$, by (\ref{primeradesig}) and
Hahn-Banach's theorem there exists $f_{in},g_{in}\in S_{X^*}$ such
that

$$
\begin{array}{c}
f_{in}(x_i+x_n)=\Vert
x_i+x_n\Vert\mathop{\geq}\limits^{(\ref{primeradesig})} 2 \left
(1-\frac{1}{n}\right )\\ \\ g_{in}(x_i-x_n)=\Vert
x_i-x_n\Vert\mathop{\geq}\limits^{(\ref{primeradesig})} 2 \left
(1-\frac{1}{n}\right )\ .
\end{array}
$$

As a consequence, for  $i\in \{1,\ldots, n\}$, we have that

$$\begin{array}{cc}
f_{in}(x_i)>1-\frac{2}{n} & f_{in}(x_n)>1-\frac{2}{n}\\ \\
g_{in}(x_i)>1-\frac{2}{n} & g_{in}(x_n)<-\left
(1-\frac{2}{n}\right )
\end{array}$$

Pick $T\in \mathbb N$ such that $1-\frac{2}{T}>1-\rho$. Then, for
 $k\geq T$, it follow that $f_{ik},g_{ik}\in S(B_{X^*},x_i,\rho)$ and so

$$\sum_{i=1}^n \alpha_i f_{ik},\sum_{i=1}^n \alpha_i g_{ik}\in \sum_{i=1}^n \alpha_i S( B_{X^*},x_i,\rho).$$

Moreover

$$\left \Vert \sum_{i=1}^n \alpha_i f_{ik}-\sum_{i=1}^n \alpha_i g_{ik} \right \Vert\geq \left \vert \sum_{i=1}^n \alpha_i
 f_{ik}(x_k)-\sum_{i=1}^n \alpha_i g_{ik}(x_k) \right \vert\geq$$

$$\geq \sum_{i=1}^n \alpha_i f_{ik}(x_k)-\sum_{i=1}^n \alpha_i g_{ik}(x_k) =\sum_{i=1}^n \alpha_i
 f_{ik}(x_k)-\sum_{i=1}^n \alpha_i g_{ik}(x_k)>$$

$$>\sum_{i=1}^n \alpha_i \left (1-\frac{2}{k} \right )+\sum_{i=1}^n \alpha_i \left (1-\frac{2}{k} \right )=2 \left (1-\frac{2}{k} \right )
 \sum_{i=1}^n \alpha_i=2-\frac{4}{k}.$$It follows that $diam\left (\sum_{i=1}^n S(
 B_{X^*},x_i,\rho)
\right)=2$.

\bigskip

ii)$\Rightarrow$i) For the converse, let $Y\subseteq X$ be a
finite-dimensional subspace, $\varepsilon\in\mathbb R^+$ and
$\delta\in\mathbb R^+$ such that $2 \delta<\varepsilon$. By
compactness of $ S_Y$ pick a $\delta-$net $\{y_1,\ldots, y_n\}$ in
$S_Y$. Let us consider the convex combination of $w^*-$slices
$$\sum_{i=1}^n \frac{1}{n}S( B_{X^*},y_i,\rho)\ \ \ {\rm whenever}\ \  0<\rho<\delta$$
and pick $0<\widehat{\rho}<\frac{\rho}{n}$.

By assumption, $diam \left (\sum_{i=1}^n \frac{1}{n} S(
B_{X^*},y_i,\rho) \right )=2$,  hence there exists
$$\sum_{i=1}^n \frac{1}{n} f_i, \sum_{i=1}^n \frac{1}{n} g_i\in
\sum_{i=1}^n \frac{1}{n} S( B_{X^*},y_i,\rho)$$ such that
$$\left \Vert \sum_{i=1}^n \frac{1}{n} f_i- \sum_{i=1}^n \frac{1}{n} g_i \right \Vert>2-\widehat{\rho} \ .$$
We put $x\in  S_X$ such that $\sum_{i=1}^n \frac{1}{n}
(f_i(x)-g_i(x))>2-\widehat{\rho}$. It follows that,
$$f_i(x)-g_i(x)>2-\rho  \ \ \ \forall i\in \{1,\ldots, n\} .$$
This implies that,
$$f_i(x)>1-\rho \ \ \mbox{and} \ \  g_i(x)<-(1-\rho),\ \ \forall i\in \{1,\ldots,
n\}.$$Furthermore, as $f_i,g_i\in S( B_{X^*},y_i,\rho)$ we have
$$
f_i(y_i)>1-\rho \ \  \mbox{and} \ \  g_i(y_i)>1-\rho ,\ \  \forall
i\in \{1,\ldots, n\}.$$ So, taking arbitrary $t\in\mathbb R_0^+$
and for $\alpha\geq 0$ one has $$\Vert t y_i+\alpha x\Vert \geq
f_i(t y_i+\alpha x)\geq t (1-\rho)+\alpha (1-\rho)=(1-\rho) (t
\Vert y_i\Vert+\vert \alpha\vert).$$ Now,  for $\alpha\leq 0$ one
has $$\Vert t y_i+\alpha x\Vert \geq g_i(t y_i+\alpha x)=t
g_i(y_i)+(-\alpha) (-g_i(x))\geq t (1-\rho)+(-\alpha) (1-\rho)=$$
$$=(1-\rho)(t \Vert y_i\Vert+\vert \alpha\vert).$$ In any case, we have
\begin{equation}\label{desigualdad particular}
\Vert t y_i+\alpha x\Vert\geq (1-\rho) (t \Vert y_i\Vert+\vert
\alpha\vert) .
\end{equation}

Pick an arbitrary  $y\in Y\setminus \{0\}$. There exists $i\in
\{1,\ldots, n\}$, such that $\left \Vert \frac{y}{\Vert
y\Vert}-y_i \right \Vert<\delta$, a hence $ \Vert y-\Vert y\Vert
y_i\Vert<\delta \Vert y\Vert$. By (\ref{desigualdad particular})
it follows $$\Vert \Vert y\Vert y_i+\alpha x\Vert\geq (1-\rho)
(\Vert y\Vert+\vert \alpha\vert).$$ Thus $$\Vert y+\alpha
x\Vert=\Vert y- \Vert y\Vert y_i+\Vert y\Vert y_i+\alpha x\Vert
\geq \Vert \Vert y\Vert y_i+\alpha x\Vert - \Vert y- \Vert y\Vert
y_i\Vert\geq$$
$$\geq (1-\rho)(\Vert y\Vert +\vert \alpha\vert )-\delta \Vert y \Vert \mathop{\geq}\limits_{\delta>\rho} (1- \delta)
 \Vert y\Vert +(1-\delta) \vert \alpha\vert-\delta \Vert y\Vert \geq$$  $$(1-2 \delta) (\Vert y\Vert+\vert \alpha\vert)
\mathop{\geq}\limits_{2 \delta<\varepsilon} (1-\varepsilon) (\Vert
y\Vert+\vert \alpha\vert)$$ So we have proved that $\forall y\in
Y, y\neq 0, \forall \alpha\in\mathbb R$ we have $$\Vert y+\alpha
x\Vert\geq (1-\varepsilon) (\Vert y\Vert+\vert \alpha\vert)$$ and
for $y=0$ is also true. We conclude that the norm $\Vert\cdot
\Vert$ is octahedral.\end{proof}

Let us observe that a Banach space $X$ satisfies that every convex
combinations of slices of $B_X$ has diameter 2 if, and only if,
every convex combination of $w^*$-slices of $B_{X^{**}}$ has
diameter 2, since $B_X$ is $w^*$-dense in $B_{X^{**}}$ and the
norm of $X^{**}$ is $w^*$-lower semicontinuous. Then the following
is a immediate consequence of the above theorem.

\begin{corollary} \lbl{octahedral-dual} Let $X$ be a Banach space. Then, every convex
combinations of slices in $B_X$ has diameter 2 if, and only if,
the norm of $X^*$ is an octahedral norm.\end{corollary}

In order to get some consequences of the above results, we recall
that a Banach space $X$ has the  Daugavet property with respect
$Y$, for some subspace $Y$ of $X^*$, if $\Vert T+I\Vert=1+\Vert
T\Vert$ for every rank one operator\break $T:X\rightarrow X$ given
by $T=x\otimes y^*$, where $x\in X$ and $y^*\in Y$. The Banach
space $X$ is said to have the almost Daugavet property if $X$
satisfies the Daugavet property with respect to some norming
subspace $Y$ of $X^*$. Finally, $X$ is said to have the Daugavet
property if $X$ satisfies the Daugavet property with respect to
$X^*$ (see \cite{KShW}).

For a Banach space $X$ satisfying the Daugavet property, it is
essentially known \cite{Sh}, that every convex combinations of
$w^*$-slices of $B_{X^*}$ has diameter 2. The next lemma shows
that the same holds for Banach spaces with the almost Daugavet
property.

\begin{lemma} \lbl{lamost-DG} Let $X$ be a Banach space satisfying the almost
Daugavet pro\-per\-ty. Then every convex combination of
$w^*$-slices in $B_{X^*}$ has diameter 2.\end{lemma}

\begin{proof} Let $Y$ a norming subspace of $X^*$ so that $X$ has the Daugavet property
with respect to $Y$. Take $x_1,\ldots ,x_n\in S_X$, $\alpha_1
,\ldots ,\alpha_n\in (0,1)$ and $\lambda_1 ,\ldots,\lambda_n>0$
with $\sum_{i=1}^n\lambda_i=1$. Let us consider the convex
combination of $w^*$-slices in $B_{X^*}$ given by
$$\sum_{i=1}^n\lambda_i S(B_{X^*},x_i,\alpha_i).$$
If $0<\varepsilon<\min_i\{\alpha_i\}$ then $$\sum_{i=1}^n\lambda_i
S(B_{X^*},x_i,\varepsilon)\subset \sum_{i=1}^n\lambda_i
S(B_{X^*},x_i,\alpha_i).$$ Pick $g\in S_Y$. Now, from \cite[Lemma
1.3]{KShW} there is $f_1\in S_{X^*}\cap
S(B_{X^*},x_1,\varepsilon)$ so that $\Vert
g+f_1\Vert>2-\varepsilon$ and then $\Vert g+\lambda_1 f_1\Vert\geq
\lambda_1+1-\varepsilon$. As $Y$ is a norming subspace of $X^*$,
we can assume that $f_1\in S_Y$. Hence $\frac{g+\lambda_1
f_1}{\Vert g+\lambda_1 f_1\Vert}\in S_Y$.

Again, from \cite[Lemma 1.3]{KShW}, there is $$f_2\in S(B_{X^*},
x_2, \frac{\varepsilon}{\Vert g+\lambda_1 f_1\Vert})\subset
S(B_{X^*}, x_2, \varepsilon)$$ such that $\Vert\frac{g+\lambda_1
f_1}{\Vert g+\lambda_1 f_1\Vert}+f_2\Vert>2-
\frac{\varepsilon}{\Vert g+\lambda_1 f_1\Vert}$. Therefore
$$\Vert\frac{g+\lambda_1 f_1}{\Vert
g+\lambda_1 f_1\Vert}+\frac{\lambda_2 }{\Vert g+\lambda_1
f_1\Vert}f_2\Vert\geq \frac{\lambda_2 }{\Vert g+\lambda_1
f_1\Vert}+1-\frac{\varepsilon}{\Vert g+\lambda_1 f_1\Vert},$$ and
so $\Vert g+\lambda_1 f_1+\lambda_2 f_2\Vert\geq
\lambda_2+\lambda_1+1-\varepsilon.$

By iterating the above argument we get $f_1, \ldots ,f_n\in S_Y$
such that\break $f_i\in S(B_{X^*}, x_i, \varepsilon)$ for every
$i$ and
$$\Vert g+\sum_{i=1}^n\lambda_i f_i\Vert\geq \sum_{i=1}^n\lambda_i
+1-\varepsilon=2-\varepsilon.$$

Now, applying the above taking $h=-\frac{\sum_{i=1}^n\lambda_i
f_i}{\Vert \sum_{i=1}^n\lambda_i f_i\Vert}$ we deduce that there
exist $h_1, \ldots ,h_n\in S_Y$ such that $h_i\in S(B_{X^*}, x_i,
\varepsilon)$ for every $i$ and
$$\Vert h+\sum_{i=1}^n\lambda_i h_i\Vert\geq \sum_{i=1}^n\lambda_i
+1-\varepsilon=2-\varepsilon.$$ Then $$diam(\sum_{i=1}^n\lambda_i
S(B_{X^*},x_i,\varepsilon)\geq \Vert\sum_{i=1}^n\lambda_i f_i
-\sum_{i=1}^n\lambda_i h_i\Vert \geq 2-2\varepsilon.$$ Hence
$diam(\sum_{i=1}^n\lambda_i S(B_{X^*},x_i,\alpha_1)\geq
2-2\varepsilon$. As $\varepsilon$ is arbitrarily small, we
conclude the proof.\end{proof}

The version of above lemma for convex combinations of slices was
proved in \cite{ALN} for Banach spaces with the Daugavet property.
The case of Banach spaces with the almost Daugavet property can be
obtained in a similar way.

\begin{lemma} \lbl{almos-DG-bis} Let $X$ be a Banach space with the almost Daugavet
porperty. Then every convex combination of slices $B_X$ has
diameter 2.\end{lemma}

It is known that for separable Banach spaces the almost Daugavet
pro\-per\-ty and having octahedral norm are equivalent
\cite{KShW}. From Theorem \ref{octahedral} and Lemmas
\ref{lamost-DG} and \ref{almos-DG-bis} we get the following

\begin{corollary} \lbl{octahedral-DG} Let $X$ be a Banach
space.\begin{enumerate}\item[i)] If $X$ has the almost Daugavet
property then the norms of $X$ and $Y$ are octahedral, where $Y$
is the norming subspace of $X^*$ such that $X$ has the Daugavet
property with respect $Y$. \item[ii)] If $X$ has the Daugavet
property then the norms of $X$ and $X^*$ are
octahedral.\end{enumerate}\end{corollary}

\begin{remark} We exhibit now an example of a Banach space $X$ failing
the Daugavet property so that the norms of $X$ and $X^*$ are
octhaedral, which disproves the converse statement of ii) in the
above Corollary. Take $X=L_1[0,1]\oplus_{\infty}\ell_1$. Now,
$L_1[0,1]$ has the Daugavet property and so, every convex
combinations of slices in $B_{L_1[0,1]}$ has diameter 2. Then
every convex combinations of slices in $B_X$ has diameter 2, from
\cite[Proposition 4.6]{ALN}, and so the norm of $X^*$ is
octahedral by Corollary \ref{octahedral-dual}. On the other hand,
$X^*=L_{\infty}[0,1]\oplus_1\ell_{\infty}$ and every convex
combination of slices in $B_{L_{\infty}}$ or $B_{\ell_{\infty}}$
has diameter 2. Therefore every convex combination of slices in
$B_{X^*}$ has diameter 2, from \cite[Theorem 2.7]{ALN}, and $X$
has octahedral norm by Theorem \ref{octahedral}. Finally it is
easy to see that $X$ fails Daugavet property, essentially because
$\ell_1$ fails Daugavet property.\end{remark}

We pass now to study the relation between Daugavet property and
octahedral norms for $JB^*$-triples. We recall that a complex
$JB^*$-triple is a complex Banach space $X$ with  a continuous
triple product $\{...\}:X\times X\times X\rightarrow X$ which is
linear and symmetric in  the outer variables, and conjugate-linear
in the middle variable, and satisfies:
\begin{enumerate}
\item For all $x$ in $X$, the mapping $y\rightarrow \{${\it
xxy}$\}$ from $X$ to $X$  is  a  hermitian operator on $X$ and has
nonnegative spectrum. \item The main identity $$
\{ab\{xyz\}\}=\{\{abx\}yz\}-\{x\{bay\}z\}+\{xy\{abz\}\} $$ holds
for all $a,b,x,y,z$  in $X$. \item $\Vert \{xxx\}\Vert =\Vert
x\Vert ^{3}$ for every $x$ in $X$.
\end{enumerate}
Concerning Condition (1) above, we also recall  that  a  bounded
linear operator $T$ on a complex Banach space $X$ is said to be
hermitian if \linebreak $\Vert \exp (irT)\Vert =1$ for every $r$
in ${\Bbb R}$. Examples of complex $JB^*$-triples are all
$C^*$-algebras under the triple product
$$\{xyz\}:=\frac{1}{2}(xy^*z+zy^*x).$$

Following \cite{IKR}, we define real $JB^*$-triples as norm-closed
real subtriples of complex $JB^*$-triples. Here, by a subtriple we
mean a subspace which is closed under triple products of its
elements. Real $JBW^*$-triples where first introduced as those
real $JB^*$-triples which are dual Banach spaces in such a way
that the triple product becomes separately $w^*$-continuous (see
\cite[Definition 4.1 and Theorem 4.4]{IKR}). Later, it has been
shown in \cite{MP} that the requirement of separate
$w^*$-continuity of the triple product is superabundant. The
bidual of every real (respectively, complex) $JB^*$-triple $X$ is
a $JBW^*$-triple under a suitable triple product which extends the
one of $X$ \cite[Lemma 4.2]{IKR} (respectively,~\cite{D}).

The following corollary characterizes the octahedral norms for
real $JB^*$-triples.

\begin{corollary} \lbl{octahedral-triples} Let $X$ be a real $JB^*$-triple. Then $X$ has
the Daugavet property if, and only if, the norm of $X$ is
octahedral.
\end{corollary}
\begin{proof} If $X$ has the Daugavet property, then from
Corollary \ref{octahedral-DG} we get that the norm $X$ is
octahedral. Assume now that the norm of $X$ is octahedral. From
Theorem \ref{octahedral}, every $w^*$-slice of $B_{X^*}$ has
diameter 2 and so, by \cite[Proposition I.1.11]{DGZ}, $X$ has no
Fr{\'e}chet differentiability points. From \cite[Theorem 3.10]{BM} $X$
has the Daugavet property.
\end{proof}

For dual of $JB^*$-triples having octahedral norm is automatic.

\begin{corollary} \lbl{octahedral-triples-dual} Let $X$ be a nonreflexive real $JB^*$-triple. Then the norm of
$X^*$ is octahedral.
\end{corollary}
\begin{proof} Let us recall that $X^*$ is a nonreflexive L-embedded Banach
space \cite[Proposition 2.2]{BLPR}. From \cite[Proposition
2.1]{BLPR}, we get that every nonempty relative weakly open set of
$B_X$ has diameter 2. Now, using the same arguments in
\cite[Proposition 2.1]{BLPR} one can prove that every convex
combination of slices in $B_{X}$ has diameter 2 and so, the norm
of $X^*$ is octahedral by Theorem \ref{octahedral}.
\end{proof}

From Corollary \ref{octahedral-DG} every Banach space with the
Daugavet property has an octahedral norm, so every convex
combination of $w^*$-slices in $B_{X^*}$ has diameter 2. On the
other hand,  if $X$ is a real $JB^*$-triple, every extreme point
of $B_{X^*}$ is actually a strongly exposed point. Indeed, given
$f\in ext(B_{X^*})$ , by \cite[Corollary 2.1]{PS} and \cite[Lemma
3.1]{BM}, assures the existence of $u\in S_{X^{**}}$ such that
$u(f)= 1$, and $u$ is a point of Fr{\'e}chet-smoothness of the norm of
$X^{**}$.  This implies that $f$ is strongly exposed by $u$ (see
\cite[Corollary I.1.5]{DGZ}). Now, the next corollary follows.

\begin{corollary} Let $X$ be a real $JB^*$-triple with the Daugavet
property. Then every convex combination of $w^*$-slices in
$B_{X^*}$ has diameter 2, but there are convex combinations of
slices in $B_{X^*}$ with diameter arbitrarily
small.
\end{corollary}

For a Banach space $X$, we define $w^*-CCS(X^*)$, respectively
$CCS(X^*)$, as the infimum of diameters of all convex combination
of $w^*$-slices, respectively slices, in $B_{X^*}$. With this
notation, the above corollary gives examples where
$w^*-CCS(X^*)=2$ and $CCS(X^*)=0$, which is the extreme case. Then
it is natural wonder when $w^*-CCS(X^*)=2$ implies $CCS(X^*)=2$.
Under some condition of $X^*$ the above holds.

\begin{proposition} Let $X$ be a Banach space and assume that
$B_{X^*}$ is the closed convex hull of the extreme point of
$B_{X^*}$. Then  $$w^*-CCS(X^*)=CCS(X^*) .$$
\end{proposition}
\begin{proof} The inequality $CCS(X^*)\leq w^*-CCS(X^*)$ is clear,
from the definitions. Now, if $S$ is a slice of $B_{X^*}$ then we
get from our assumption that $S\cap ext(B_{X^*})\neq \emptyset$,
by convexity. Hence there is $S^*$ a $w^*$-slice of $B_{X^*}$ so
that $S^*\subset S$, by Choquet's Lemma (see \cite[Lemma
3.40]{FHHMPZ}). Therefore every convex combination of slices in
$B_{X^*}$ contains a convex combination of $w^*$-slices in
$B_{X^*}$ and we are done.
\end{proof}

Observe that the above proposition holds in  particular for Banach
spaces not containing isomorphic copies of $\ell_1$ \cite{FHHMPZ}.
On the other hand , it is known (see \cite[Theorem II.4, Remark
II.5]{G} and \cite[Proposition 4 and Corollary 6]{D}) that a
Banach space containing $\ell_1$ if and only if has an equivalent
octahedral norm if and only if there is a equivalent norm such
that $w^*-CCS(X^*)=2$. The natural question then is to know if a
Banach space containing $\ell_1$ can be equivalently renormed so
that $CCS(X^*)=2$. In \cite[Remark II.5, 3)]{G}, it is asked if
every Banach space containing $\ell_1$ has an equivalent norm so
that the corresponding bidual norm is octahedral. From Theorem
\ref{octahedral} we deduce that this last question is equivalent
to asking if every Banach space containing $\ell_1$ can be
equivalently renormed so that $CCS(X^*)=2$. Our next result can be
seen like a partial answer to the above question.

\begin{proposition} \lbl{convex-w} Let $X$ be a separable Banach space containing
a subspace isomorphic to $\ell_1$. Then for every $\varepsilon>0$
there is an equivalent norm in $X$ such that every convex
combination of slices of the new unit ball of $X^{*}$ has
diameter, at least, $2-\varepsilon$.
\end{proposition}

In order to prove the above proposition we need a couple of
lemmas.

\begin{lemma}\label{simetrico} Let $X$ be a Banach space and $C$ is a convex,
$w^*$-compact subset of $B_{X^*}$ such that every convex
combination of slices in $C$ has diameter 2. Then  the set
$K=co(C\cup -C)$ is a $w^*$-compact convex subset of $B_{X^*}$
such that every convex combination of slices in $K$ has diameter
2.
\end{lemma}
\begin{proof} As $C$ is a $w^*$-compact and convex subset, then
$K$ is also $w^*$-compact and convex. This is a consequence from
the fact that $$K=\{\lambda a-(1-\lambda)b:\lambda\in[0,1],\
a,b\in C\}.$$ Pick $S_1,\ldots ,S_n$ slices of $K$ and
$\lambda_1,\ldots ,\lambda_n>0$ with $\sum_{i=1}^n \lambda_i=1$.
Let $A=\{i\in\{1,\ldots ,n\}: S_i\cap C\neq \emptyset\}$ and
$B:=\{1,\ldots ,n\}\setminus A$. Let's observe that every slice of
$K$ has to be nonempty intersection with $C$ or $-C$.

Now we have that
$$\Lambda:=\sum_{i\in A}\lambda_i(S_i\cap C)+\sum_{i\in B}\lambda_i(S_i\cap
(-C))\subset \sum_{i=1}^n\lambda_iS_i,$$ and then
$$\Lambda-\Lambda=\sum_{i\in A}\lambda_i(S_i\cap C)+\sum_{i\in B}\lambda_i(S_i\cap
(-C))-\sum_{i\in A}\lambda_i(S_i\cap C)-\sum_{i\in
B}\lambda_i(S_i\cap (-C))=$$
$$\sum_{i\in A}\lambda_i(S_i\cap C)+\sum_{i\in B}\lambda_i(-S_i\cap
C)-(\sum_{i\in A}\lambda_i(S_i\cap C)+\sum_{i\in
B}\lambda_i(-S_i\cap C))=D-D,$$ where $D=\sum_{i\in
A}\lambda_i(S_i\cap C)+\sum_{i\in B}\lambda_i(-S_i\cap C)$ is a
convex combination of slices in $C$. From the hypothesis, we have
that $diam(D)$=2, hence we get that $diam(\Lambda)=2$ and so
$diam(\sum_{i=1}^n\lambda_iS_i)=2$.\end{proof}

\begin{lemma}\label{renorming}
Let $(X,\Vert\cdot\Vert)$ be a Banach space and $C\subset B_X$ an
absolutely convex and closed subset satisfying that every convex
combination of slices has $\Vert\cdot\Vert$-diameter 2. Then for
every $\varepsilon>0$ there is   an equivalent norm
$\vert\cdot\vert$ in $X$ such that every convex combination of
slices in $B_{(X, \vert\cdot\vert)}$ has
$\vert\cdot\vert$-diameter at least $2-\varepsilon$.
\end{lemma}
\begin{proof}
Pick an arbitrary $\varepsilon>0$ and we put $\eta\in \real^+$
such that $\frac{2-2\eta}{1+\eta}>2-\varepsilon$. Consider
$\vert\cdot\vert$ the equivalent norm in $X$ whose unit ball is
$$B_{\vert\cdot\vert}:=\overline{C+\eta B_X}\  .$$ Now choose $n\in \natu$,
$\beta_1,\ldots,\beta_n\in (0,1)$, $\lambda_1,\ldots ,\lambda_n>0$
with $\sum_{i=1}^n \lambda_i=1$ and $f_1,\ldots ,f_n\in
S_{(X,\vert\cdot\vert)^*}$. Let us see that the convex combination
of slices $\sum_{i=1}^n \lambda_i
S(B_{\vert\cdot\vert},f_i,\beta_i)$ has diameter $2-\varepsilon$.
We put, for $i\in \{1,\ldots ,n\}$, $\gamma_i:=\sup_{C}f_i$ and
$\delta_i:= \sup_{B_X}f_i$, then we have that $\gamma_i +\eta
\delta_i=1$. We consider $\rho \in \real$ such that $0<\rho <\min
\{\beta _i,\gamma_i,\delta_i,\beta _i \eta ,\gamma_i \eta
,\delta_i \eta :i=1,\ldots ,n\}$. As a consequence, we have that
for every $1\leq i\leq n$ one has
$$ S(C,f_i, \frac{\rho}{2} ) +\eta S(B_X,f_i,
\frac{\rho}{2\eta }) \subset
 S( B_{\vert\cdot\vert},f_i,\rho).$$So $\sum_{i=1}^n \lambda_i   S(C,f_i,\frac{ \rho}{2}
)+\lambda_i\eta S  (  B_X,f_i, \frac{ \rho}{2 \eta} ) $ is
contained in $\sum_{i=1}^n \lambda_i
S(B_{\vert\cdot\vert},f_i,\beta_i)$. Now, as
$$\Delta:=\sum_{i=1}^n \lambda_i  S  (C,f_i,\frac{\rho}{2})$$
is a convex combination of slices of $C$, we get that
$\Vert\cdot\Vert-diam(\Delta)=2$. Moreover $$\Gamma:=\sum_{i=1}^n
\lambda_i  S(B_X,f_i, \frac{\rho}{2 \eta})$$ is a subset of $B_X$,
and hence $\Vert\cdot\Vert-$diameter is at most $2$. Hence
$$\Vert\cdot\Vert-diam(\Delta +\eta \Gamma)\geq 2-2\eta$$
and so $$\Vert\cdot\Vert-diam(\sum_{i=1}^n\lambda_i
S(B_{\vert\cdot\vert},f_i,\beta_i))\geq 2-2\eta.$$ Finally, from
$B_{\vert\cdot\vert}\subset (1+\eta)B_X$ we deduce that
$$\vert\cdot\vert-diam(\sum_{i=1}^n\alpha_i
S(B_{\vert\cdot\vert},x_i^*,\beta_i))\geq
\frac{2-2\eta}{1+\eta}>2-\varepsilon .$$\end{proof}

\begin{proof} of Proposition \ref{convex-w}. Assume that $X$ contains a subspace isometric to
$\ell_1$ and fix $\varepsilon>0$. From \cite[Theorem 2]{DGH} we
know $C(\Delta)$ is isometric to a quotient space of $X$, where
$\Delta=\{0,1\}^{\natu}$ is the Cantor set. Now $X^*$ contains a
subspace $Z$ isometric to $C(\Delta)^*$. Furthermore, $Z$ is
$w^*$-closed in $X^*$ and the weak-star topology of $X^*$ on $Z$
is the weak-star topology of $C(\Delta)^*$ on $Z$. Now, from
\cite[Theorem 4.6]{ScSeWe}, there is a $w^*$-compact and convex
subset $C$ of $S_{Z}$ so that every convex combination of slices
in $C$ has diameter 2. From lemma \ref{simetrico} we get that
$K:=co(C\cup (-C))$ is a $w^*$-compact and absolutely  convex
subset of $B_{X^*}$ such that every convex combination of slices
in $K$ has diameter 2. Finally, from lemma \ref{renorming} we get
an equivalent norm in $X^*$ and the new unit ball $B$ in $X^*$
satisfies  that every convex combination of slices in $B$ has
diameter $2-\varepsilon$. As we have, for some $\eta>0$, that
$B=co(K+\eta B_{X^*})$ is $w^*$-closed the new norm in $X^*$ is a
dual norm and the proof is complete. \end{proof}

We don't know if the above proposition is valid for nonseparable
Banach spaces containing $\ell_1$-copies.

We recall that a Banach space is said to be strongly regular if
every closed, bounded and convex subset of $X$ contains convex
combinations of slices with diameter arbitrarily small. Similarly,
$X^*$ is said to be $w^*$-strongly regular if  $w^*$-compact and
convex subset of $X^*$ contains convex combinations of
$w^*$-slices with diameter arbitrarily small. We refer to
\cite{GGMS} for background about these topics. It is known that
$X^*$ is strongly regular if, and only if, $X^*$ is $w^*$-strongly
regular which is equivalent to $X$ does not containing isomorphic
copies of $\ell_1$ \cite[Corollary VI.18]{GGMS} and, from
\cite{G}, equivalent to $X$ having an equivalent octahedral norm.
With these known facts joint to Theorem \ref{octahedral} and
Proposition \ref{convex-w} we get the following final

\begin{corollary} Let $X$ be a Banach space. Consider the
following assertions: \begin{enumerate}\item[i)] $X$ contains
isomorphic subspaces to $\ell_1$.\item[ii)] $X^*$ fails to be
strongly regular. \item[iii)] $X^*$ fails to be $w^*$-strongly
regular. \item[iv)] $X$ has an equivalent octahedral norm.
\item[v)] $X$ has an equivalent norm so that every convex
combination of $w^*$-slices in the new unit ball of $X^*$ has
diameter 2. \item[vi)] For every $\varepsilon>0$ there is an
equivalent norm in $X$ so that every convex combination of slices
in the new unit ball of $X^*$ has diameter
$2-\varepsilon$.\end{enumerate} Then the statements i), ii), iii),
 iv) and v) are equivalent and, if $X$ is separable, the six statements
are equivalent.\end{corollary}

Now the aforementioned question about if every Banach space
containing $\ell_1$-copies can be equivalent renormed so that the
corresponding bidual norm is octahedral posed in \cite[Remark
II.5, 3)]{G} is equivalent to wonder if one can get the
equivalence in the above corollary with $\varepsilon=0$, which
seems highly non trivial.

Finally, we remark that the above question has an affirmative
answer if $X$ is a Banach space containing a complemented
isomorphic copy of $\ell_1$. Indeed, we can assume that $X$
contains $Y$ an isometric and complemented copy of $\ell_1$. Then
there is a linear and continuous projection $P:X\rightarrow Y$.
Let's define $\vert x\vert=\Vert P(x)\Vert +\Vert x-P(x)\Vert$ for
every $x\in X$. Now $\vert\cdot\vert$ is an equivalent norm in $X$
such that $(X,\vert\cdot\vert)^*=Y^*\oplus_{\infty}(Ker\ P)^*$. As
$Y^*$ is isometric to $\ell_{\infty}$, we get that every convex
combination of slices in $B_{Y^*}$ has diameter 2. From
\cite[Proposition 4.6]{ALN}, we deduce that every convex
combination of slices in $B_{(X,\vert\cdot\vert )^*}$ has diameter
2. Finally, from Corollary \ref{octahedral-dual} we are done.

\bigskip

{\bf Acknowledgements.} We would like to thank  professor Gilles
Godefroy for notice us about the relation between octahedral norms
and convex combination of slices and for kindly answering our
inquiries.


\bigskip

\begin{thebibliography}{999999}

\bibitem{ALN}
T.A. Abrahansen,  V. Lima and O. Nygaard,  {\it Remarks on
diameter two properties}, J. Convex Anal. {\bf 20} (2013), to
appear.
%
\bibitem{BLPR}
J. Becerra, G. L\'{o}pez, A. Peralta and A. Rodr\'{\i}guez, {\it
Relatively weakly open sets in closed balls of Banach spaces, and
real $JB^*$-triples of finite rank,}  Math.  Ann. \textbf{330}
(2004), 45-58.
%
\bibitem{BM}
J. Becerra and M. Mart\'in, {\it The Daugavet property of
$C^*$-algebras, $JB^*$-triples, and of their isometric preduals},
J. Functional Analysis {\bf 224} (2005), 316-337.
%
\bibitem{D}
R. Deville, {\it A dual characterization of the existence of small
combinations of slices}, Bull. Austral. Math. Soc. {\bf 37}
(1988), 113-120.
%
\bibitem{DGZ}
R. Deville, G. Godefroy and V. Zizler, {\it Smoothness  and
renormings in Banach spaces}.   Pitman Monographs and Surveys in
Pure and Applied Math. {\bf 64},  1993.
%
\bibitem{DGH}
S. J. Dilworth, M. Girardi and J. Hagler. {\it Dual Banach spaces
which contain isometric copy of} $L_1$.   Bull. Polish Acad. Sci.
Math. {\bf 48} (2000),  1{\^u}12.
%
\bibitem{D}
S. Dineen, The second dual  of a $JB^*$-triple  system.  In {\it
Complex Analysis, Functional Analysis and Approximation Theory}
(ed. by  J.  M\'{u}gica), 67-69, North-Holland Math. Stud. {\bf
125},  North-Holland,  Amsterdam-New  York, 1986.

%
\bibitem{FHHMPZ}
M. Fabian, P. Habala, P. H\'ajek, V. Montesinos, J. Pelant and  V.
Zizler, Functional Analysis and Infinite-dimensional Geometry ,
{\it CM Books in Mathematics.} Springer-Verlag. Berlin 2001.
%
\bibitem{GGMS}
N. Ghoussoub, G. Godefroy, B. Maurey and  W. Schachermayer, {\it
Some topological and geometrical structures in Banach spaces,}
Mem. Amer. Math. Soc. {\bf 378}, 1987.
%
\bibitem{G}
G. Godefroy, {\it Metric characterization of first Baire class
linear forms and octahedral norms,} Studia Math. {\bf 95} (1989),
1-15.
%
\bibitem{IKR} J. M. Isidro, W. Kaup, and A. Rodr\'iguez, {\it On real forms of
$JB^*$-triples}  Manuscripta Math. {\bf 86} (1995), 311-335.
%
\bibitem{KShW}
 V. Kadets, V. Shepelska, and D. Werner, {\it Thickness of the unit sphere, $\ell_1$-types, and the Daugavet property,}
 Houston J. Math. \textbf{37} (2011), 867-878.
%
\bibitem{MP} J. Mart\'inez and A. M. Peralta, {\it
Separate weak$^*$-continuity of the triple product in dual real
$JB^*$-triples},  Math. Z. {\bf 234}, 635-646.
%
\bibitem{NyWe}
O. Nygaard and D. Werner, {\it Slices in the unit ball of a
uniform algebra,} Arch. Math. \textbf{76} (2001), 441-444.
%
\bibitem{PS}
A. M. Peralta and L. L. Stach\'o, {\it Atomic decomposition of
real $JBW^*$-triples}, Quart. J. Math. \textbf{52} (2001), 79-87.

%
\bibitem{ScSeWe}
W. Schachermayer, A. Sersouri and E. Werner, {\it Moduli of
nondentability and the Radon-Nikod{\'y}m property in Banach spaces},
Israel J. Math. \textbf{65} (3) (1989),  225-257.
%
\bibitem{Sh}
 R.V. Shvydkoy, {\it Geometric aspects of the Daugavet
property}, J. Funct. Anal. \textbf{176} (2000), 198-212.
%

\end{thebibliography}
\end{document}